\documentclass[12pt, reqno, a4paper]{amsart}

\usepackage{amssymb}
\usepackage{hyperref}

\usepackage[english]{babel}

\theoremstyle{plain}

\newtheorem*{corollary}{Corollary}
\newtheorem{lemma}{Lemma}
\newtheorem{theorem}{Theorem}

\theoremstyle{remark}
\newtheorem*{remark}{Remark}

\theoremstyle{definition}

\newtheorem{example}{Example}

\DeclareMathOperator{\Der}{Der}
\DeclareMathOperator{\id}{id}

\DeclareMathOperator{\ad}{ad}

\DeclareMathOperator{\End}{End}
\DeclareMathOperator{\Aut}{Aut}

\DeclareMathOperator{\Hom}{Hom}

\begin{document}

\title{Structure of $H$-(co)module
Lie algebras}

\author{A.\,S.~Gordienko}

\address{Memorial University of Newfoundland, St. John's, NL, Canada}
\email{asgordienko@mun.ca}

\begin{abstract}
Let $L$ be a finite dimensional Lie algebra over a field of characteristic $0$. Then by the original Levi
theorem, $L = B \oplus R$ where $R$ is the solvable radical and $B$ is some maximal semisimple subalgebra.
We prove that if $L$ is an $H$-(co)module algebra  for a finite dimensional (co)semisimple Hopf algebra $H$,
then $R$ is $H$-(co)invariant and $B$ can be chosen to be $H$-(co)invariant too.
Moreover, the nilpotent radical $N$ of $L$ is $H$-(co)invariant and there exists an $H$-sub(co)module $S\subseteq R$ such that $R=S\oplus N$ and $[B,S]=0$.
In addition, the $H$-(co)invariant analog of the Weyl theorem is proved.
 In fact, under certain
conditions, these results hold for an $H$-comodule Lie algebra $L$, even if $H$ is infinite dimensional.
In particular, if $L$ is a Lie algebra graded by an arbitrary group $G$, then $B$ can be chosen to be graded,
and if $L$ is a Lie algebra with a rational action of a reductive affine algebraic
group $G$ by  automorphisms, then $B$ can be chosen to be $G$-invariant.
 Also we prove that every finite dimensional semisimple $H$-(co)module Lie algebra over a field of characteristic $0$ is a direct sum of its minimal $H$-(co)invariant ideals.
\end{abstract}


\thanks{
Supported by post doctoral fellowship
from Atlantic Association for Research
in Mathematical Sciences (AARMS), Atlantic Algebra Centre (AAC),
Memorial University of Newfoundland (MUN), and
Natural Sciences and Engineering Research Council of Canada (NSERC)}

\maketitle

The applications of Lie and associative algebras
with an additional structure, e.g. graded, $H$-(co)module, or $G$-algebras, gave rise to the studies
of the objects and decompositions that have nice properties with respect to these structures.
One of the applications of invariant decompositions is in the combinatorial theory of graded, $G$- and $H$-polynomial
identities~\cite{AljaGia, AljaGiaLa, GiaLa, GiaSheZai, ZaiGia, ASGordienko2, ASGordienko3, ZaiLie}.
%Note that the existence part of the Wedderburn~--- Mal'cev
%theorem was originally proved by J.~Wedderburn and the uniqueness up to an inner automorphism
%was proved by A.\,I.~Mal'cev~\cite{Malcev}. Only the existence part is discussed in this paper. However,
%in order to avoid confusions with the other Wedderburn's theorems, we use the name of the whole theorem.

The Levi theorem is one of the main results of the structure Lie theory, as
well as the Wedderburn~--- Mal'cev theorem is one of the central results
in the structure ring theory. In 1957 E.J.~Taft proved~\cite{Taft} the $G$-invariant Levi and Wedderburn~--- Mal'cev
theorems 
for $G$-algebras with an action of a finite group $G$ by automorphisms and anti-automorphisms.
Due to a well-known duality between $G$-gradings and $G$-actions, Taft's result implies
graded decompositions of algebras graded by a finite Abelian group $G$. 
The study of Wedderburn decompositions for $H$-module algebras was started by A.\,V. Sidorov~\cite{Sidorov}
in 1986. In 1999 D. \c Stefan and F. Van Oystaeyen~\cite{SteVanOyst} proved
 the $H$-coinvariant Wedderburn~--- Mal'cev theorem for finite dimensional $H$-comodule associative algebras,
 where $H$ is
a Hopf algebra with an $\ad$-invariant left integral $t\in H^*$
such that $t(1)=1$.
In particular, they proved the $H$-(co)invariant Wedderburn~--- Mal'cev theorem for finite dimensional semisimple $H$
over a field of characteristic $0$, the graded Wedderburn~--- Mal'cev theorem for  any grading group provided that the Jacobson radical is graded too,
and the $G$-invariant Wedderburn~--- Mal'cev theorem for associative algebras with a rational action of a reductive algebraic group $G$ by automorphisms.
 The graded Levi theorem for finite dimensional Lie algebras over an algebraically closed field of characteristic $0$, graded by a finite group, was proved by D.~Pagon, D.~Repov\v s, and M.V.~Zaicev~\cite{PaReZai} in 2011.
 
In this paper we prove the $H$-coinvariant Levi theorem in the case when the Hopf algebra $H$
has an $\ad$-invariant left integral $t\in H^*$
such that $t(1)=1$ (Theorem \ref{TheoremHcoLevi}). As a consequence we obtain the $H$-invariant Levi theorem for $H$-module Lie algebras for a finite dimensional semisimple Hopf algebra $H$ (Theorem \ref{TheoremHLevi}), the graded Levi theorem for an arbitrary grading group (Theorem~\ref{TheoremGradLevi}),
and the $G$-invariant Levi theorem for Lie algebras with a rational action of a reductive algebraic group $G$ (Theorem~\ref{TheoremAffAlgGrLevi}).

An important condition in the invariant Levi and Wedderburn~--- Mal'cev theorems is 
the stability of the radicals. In the case of $G$-algebras the stability is clear since
 the radicals are invariant under automorphisms and anti-automorphisms. 
In 1984 M.~Cohen and S.~Montgomery~\cite{CohenMont} proved
that the Jacobson radical of a $G$-graded associative algebra is graded if $|G|^{-1}$
belongs to the base field. In 2001 V.~Linchenko~\cite{LinchenkoJH} proved
the stability of the Jacobson radical of a finite dimensional $H$-module associative algebra
over a field of characteristic $0$ for a finite dimensional semisimple Hopf algebra~$H$.
This result was later generalized by V.~Linchenko, S.~Montgomery and L.W.~Small~\cite{LinMontSmall}.
In 2011 D.~Pagon, D.~Repov\v s, and M.V.~Zaicev~\cite[Proposition 3.3 and its proof]{PaReZai} proved that
the solvable radical of a finite dimensional Lie algebra over an algebraically closed field 
of characteristic $0$, graded by any group, is graded.

 In this paper we prove that the solvable and the nilpotent radicals of an $H$-(co)module Lie algebra over a field of characteristic $0$ are $H$-(co)invariant for any finite dimensional (co)semisimple Hopf algebra $H$ (Theorem~\ref{TheoremRadicals}).
 
 In~\cite{ZaiLie} M.\,V.~Zaicev used the decomposition $L=B\oplus S\oplus N$ (direct sum of subspaces)
where $B$ is a maximal semisimple subalgebra of a finite dimensional Lie algebra $L$
over a field of characteristic $0$, $N$ is the nilpotent
radical of $L$, $S\oplus N$ is the solvable radical of $L$, $[B, S]=0$,
in order to prove the analog of S.\,A.~Amitsur's conjecture  on codimension growth of polynomial identities
 for Lie algebras. In~\cite{ASGordienko2} the author used a $G$-invariant
 decomposition $L=B\oplus S\oplus N$ in order to prove the analog of S.\,A.~Amitsur's conjecture  for polynomial $G$-identities
 and graded identities of Lie algebras. (In the graded case the grading group was required to be Abelian.)
In Section~\ref{SectionHLBSN} we prove the existence of an $H$-(co)invariant
decomposition  $L=B\oplus S\oplus N$. This is needed in the
study of polynomial $H$-identities and graded identities for a non-Abelian grading group
(see also~\cite{ASGordienko3}). The main tool in this proof is the $H$-(co)invariant analog of the Weyl theorem
(Section~\ref{SectionHWeyl}).

In Section~\ref{SectionHSS} we prove that every finite dimensional semisimple $H$-(co)module Lie algebra over a field of characteristic $0$ is a direct sum of its minimal $H$-(co)invariant ideals.
In addition to its own interest, this result is needed in the
study of polynomial $H$-identities as well.

\section{Introduction} 

\subsection{Graded spaces and Lie algebras}

Let $V$ be a vector space over a field $F$ and $G$ be a group.
We say that $V$ is \textit{$G$-graded}
if there is a fixed decomposition $V=\bigoplus_{g \in G} V^{(g)}$.
Let $V=\bigoplus_{g \in G} V^{(g)}$ and $W=\bigoplus_{g \in G} W^{(g)}$
be two graded vector spaces. A linear map $\varphi \colon V \to W$
is \textit{graded} if $\varphi(V^{(g)}) \subseteq W^{(g)}$
for all $g \in G$.

Let $L=\bigoplus_{g \in G} L^{(g)}$ (direct sum of subspaces) be a Lie algebra. We say that $L$ is \textit{graded} if $[L^{(g)},L^{(h)}]\subseteq L^{(gh)}$ for all $g,h \in G$.  A subspace $W \subseteq L$
is \textit{graded} if $W=\bigoplus_{g\in G} (L^{(g)} \cap W)$.

\begin{example}[\cite{PaReZai}]
Consider $L = \left\lbrace\left(\begin{array}{cc} \mathfrak{gl}_2(F) & 0 \\ 0 & \mathfrak{gl}_2(F) \end{array}\right)\right\rbrace \subseteq \mathfrak{gl}_4(F)$.
Then $L$ is an $S_3$-graded Lie algebra where
$S_3$ is the third symmetric group with the unit $e$,
$$L^{(e)} = \left\lbrace\left(\begin{array}{cccc}
 \alpha & 0 &  0 & 0 \\
  0 & \beta & 0 & 0 \\
   0 & 0 & \gamma & 0 \\
    0 & 0 & 0 & \lambda
 \end{array}\right)\right\rbrace,\quad
L^{\bigl((12)\bigr)} = \left\lbrace\left(\begin{array}{cccc} 0 & \beta & 0 & 0 \\
 \alpha & 0 &  0 & 0 \\
  0 & 0 & 0 & 0 \\
   0 & 0 & 0 & 0
 \end{array}\right)\right\rbrace,$$ $$
L^{\bigl((23)\bigr)} = \left\lbrace\left(\begin{array}{cccc}
  0 & 0 & 0 & 0 \\
   0 & 0 & 0 & 0 \\
  0 & 0 & 0 & \lambda \\
   0 & 0 & \gamma & 0 \\
 \end{array}\right)\right\rbrace,$$
 the other components are zero, $\alpha,\beta,\gamma,\lambda \in F$. 
\end{example}

\subsection{Integrals on Hopf algebras}
The related notions to graded Lie algebras are the ones of $H$-module and $H$-comodule Lie algebras for a Hopf algebra $H$. We refer the reader to~\cite{Abe, Danara, Montgomery, Sweedler} for an account
  of Hopf algebras.   

   Recall that $t \in H^*$ is a \textit{left integral on $H$} if $t(a_{(2)})a_{(1)}= t(a)1$ for all $a \in H$.
Here we use Sweedler's notation $\Delta(a)=a_{(1)}\otimes a_{(2)}$.
      We say that a left integral is \textit{$\ad$-invariant}
 if $t(a_{(1)}\ b\ S(a_{(2)}))=\varepsilon(a)t(b)$ for all $a,b \in H$.
  Only the results of Section~\ref{SectionHSS} are proved for an arbitrary $H$.
 In the other sections we assume that there exists an $\ad$-invariant left integral
 $t \in H^*$ such that $t(1)=1$.
 Now we list three main examples~\cite{SteVanOyst} of such Hopf algebras~$H$.

First, we notice that the existence of an integral $t \in H^*$, such that $t(1)=1$,
 is equivalent to the cosemisimplicity of $H$ (see e.g.~\cite[Exercise~5.5.9]{Danara}).

\begin{example}\label{ExampleHSS}
Let $H$ be a finite dimensional (co)semisimple Hopf algebra
over a field $F$ of characteristic $0$. 
Then there exists an $\ad$-invariant left integral $t \in H^*$ such that $t(1)=1$.
\end{example}
\begin{proof}
By the Larson~--- Radford theorem (see e.g. \cite[Theorem 7.4.6]{Danara}),
 $H$ is semisimple if and only if it is cosemisimple. The third equivalent 
condition is that $S^2 = \id_H$. Since $H$ is cosemisimple, there exists a left integral $t \in H^*$ such that $t(1)=1$. Every integral on such algebra $H$ is cocommutative (see e.g.~\cite[Exercise~7.4.7]{Danara}), i.e. $t(ab)=t(ba)$ for all $a,b \in H$.  Thus $$t(a_{(1)}\ b\ S(a_{(2)}))=t(b (Sa_{(2)})a_{(1)})=
t\Bigl(b\ S\bigl((Sa_{(1)})a_{(2)}\bigr)\Bigr)=\varepsilon(a)t(b)\text{ for all }a,b \in H$$ and $t$ is $\ad$-invariant.\end{proof}

\begin{example}\label{ExampleFG} Let $G$ be any group.
Denote by $FG$ the group algebra of $G$. Then $FG$ is a Hopf algebra
with the comultiplication $\Delta(g)=g\otimes g$, the counit $\varepsilon(g)=1$,
and  the antipode $S(g)=g^{-1}$, $g \in G$. Consider $t \in (FG)^*$,
$t(g)=\left\lbrace\begin{array}{ccc} 0 & \text{ if } & g\ne 1, \\ 
1 & \text{ if } & g = 1. \end{array}\right.$
Then $t$ is an $\ad$-invariant left integral on $FG$. Note that $t(1)=1$.
\end{example}

\begin{example}\label{ExampleAffAlgGr}
Let $G$ be an
affine algebraic group over a field $F$.
Denote by $\mathcal O(G)$ the coordinate algebra of
$G$. Then $\mathcal O(G)$ is a Hopf algebra
where the comultiplication $\Delta \colon \mathcal O(G) \to \mathcal O(G) \otimes \mathcal O(G)$
is dual to the multiplication $G \times G \to G$,
the counit $\varepsilon \colon  \mathcal O(G) \to F$
is defined by $\varepsilon(f)=f(1_G)$, and the antipode $S \colon 
\mathcal O(G) \to \mathcal O(G)$ is dual to the map $g \to g^{-1}$, $g \in G$.
  If $F$ is algebraically closed of characteristic $0$ and $G$ is reductive,
then there exists an $\ad$-invariant left integral $t \in \mathcal O(G)^*$ such that $t(1)=1$.
\end{example}
\begin{proof}
All rational representations of $G$ (see e.g.~\cite{Nagata})
are completely reducible. Hence by~\cite[Theorem 4.2.12]{Abe}, $\mathcal O(G)$ is cosemisimple.
 Thus there exists a left integral $t \in H^*$ such that $t(1)=1$. Since $H$ is commutative,
$t$ is $\ad$-invariant.
\end{proof}

We conclude the subsection with an example of a Hopf algebra that does not have
nonzero integrals.
\begin{example}\label{ExampleUnivEnv}
Let $L$ be a Lie algebra over a field $F$. The universal enveloping algebra $U(L)$
is a Hopf algebra where $\Delta(v)=v\otimes 1 + 1 \otimes v$, $\varepsilon(v)=0$,
$S(v)=-v$ for all $v \in L$. If $L \ne 0$, $F$ is of characteristic $0$,
and $t \in U(L)^*$ is a left integral of $U(L)$, then $t=0$.
\end{example}
\begin{proof}
By Poincar\'e~--- Birkhoff~--- Witt theorem, it 
 is sufficient to show that $t(v_1^{m_1}\ldots v_k^{m_k})=0$
for all linearly independent $v_1, \ldots, v_k \in L$ and all $m_1, \ldots, m_k \geqslant 0$,
$k \geqslant 0$. We fix $v_1, \ldots, v_k \in L$, introduce the degree lexicographic
ordering $\prec$ on $k$-tuples $(m_1, \ldots, m_k)$, and prove the assertion by induction.
First, we notice that $t(v)1 = t(1)v+t(v)1$ for all $v \in L$. Hence $t(1)=0$
and the base of the induction is proved. Suppose $m_1, \ldots, m_k \geqslant 0$.
Denote $\Lambda := \lbrace (\ell_1, \ldots, \ell_k) \mid 0 \leqslant \ell_i \leqslant m_i,
\ 1 \leqslant i \leqslant k-1;\ 0 \leqslant \ell_k \leqslant m_k+1 \rbrace$
and $$\Lambda_1 := \Lambda\backslash \lbrace (m_1, \ldots, m_k+1),  (m_1, \ldots, m_k) \rbrace
= \lbrace  (\ell_1, \ldots, \ell_k) \in \Lambda \mid (\ell_1, \ldots, \ell_k) \prec (m_1, \ldots, m_k)\rbrace.$$
Then $$t(v_1^{m_1}\ldots v_k^{m_k+1})1=t\left((v_1^{m_1}\ldots v_k^{m_k+1})_{(2)}\right)
(v_1^{m_1}\ldots v_k^{m_k+1})_{(1)}=$$ $$\sum_{(\ell_1, \ldots, \ell_k)\in\Lambda} \tbinom{m_1}{\ell_1}\tbinom{m_2}{\ell_2}\ldots \tbinom{m_{k-1}}{\ell_{k-1}}
\tbinom{m_k+1}{\ell_k}\ t(v_1^{\ell_1}\ldots v_k^{\ell_k})
\ v_1^{m_1-\ell_1}\ldots v_k^{(m_k+1)-\ell_k}
=$$
$$t(v_1^{m_1}\ldots v_k^{m_k+1})1+(m_k+1)t(v_1^{m_1}\ldots v_k^{m_k})v_k+
$$
$$\sum_{(\ell_1, \ldots, \ell_k)\in\Lambda_1} \tbinom{m_1}{\ell_1}\tbinom{m_2}{\ell_2}\ldots \tbinom{m_{k-1}}{\ell_{k-1}}
\tbinom{m_k+1}{\ell_k}\ t(v_1^{\ell_1}\ldots v_k^{\ell_k})
\ v_1^{m_1-\ell_1}\ldots v_k^{(m_k+1)-\ell_k}$$ $$
=t(v_1^{m_1}\ldots v_k^{m_k+1})1+(m_k+1)t(v_1^{m_1}\ldots v_k^{m_k})v_k+0 $$
by the induction hypothesis.
 Hence $t(v_1^{m_1}\ldots v_k^{m_k})=0$.
\end{proof}
  
\subsection{$H$-comodule Lie algebras}
Let $L$ be a Lie algebra over a field $F$. Suppose $L$ is a right $H$-comodule for some Hopf algebra
$H$. Denote by $\rho \colon L \to L \otimes H$ the corresponding comodule map. We say that $L$ is an
\textit{$H$-comodule Lie algebra} if $\rho([a,b])=[a_{(0)},b_{(0)}] \otimes a_{(1)}b_{(1)}$
for all $a,b \in L$. Here we use
Sweedler's notation $\rho(a)=a_{(0)}\otimes a_{(1)}$.

\begin{example}\label{ExampleGraded} 
Let $L=\bigoplus_{g \in G} L^{(g)}$ be a Lie
algebra over a field $F$ graded by a group $G$. Then $L$ is an $FG$-comodule algebra where
$\rho(a^{(g)})=a^{(g)} \otimes g$ for all $g \in G$ and $a^{(g)} \in L^{(g)}$. 
\end{example}

Another example of $H$-comodule algebras arises when
an affine algebraic group acts on a Lie algebra.

\begin{example}\label{ExampleRatAffAlgGrAction} Let $L$ be a Lie algebra
over a field $F$ and let $G$ be an affine algebraic group. Suppose $L$ is endowed with a \textit{rational action} of $G$ by automorphisms, i.e. there is a fixed homomorphism $G \to \Aut(L)\subseteq GL(L)$
such that for some basis $e_1, \ldots, e_m$ of $L$
we have $g e_j =\sum_{i=1}^m \omega_{ij}(g) e_i$ where $\omega_{ij}$ are polynomials in the coordinates of $g \in G$. Then $L$ is an $\mathcal O(G)$-comodule algebra where $\rho(e_j)=\sum_{i=1}^m e_i \otimes \omega_{ij}$,
$1 \leqslant j \leqslant m$,
and $g a = a_{(1)}(g) a_{(0)}$, $g \in G$.
Furthermore, each $\mathcal O(G)$-subcomodule of $L$ is a $G$-invariant subspace and vice versa. 
\end{example}
\begin{proof} Suppose $[e_i, e_k] = \sum_{j=1}^m c^{j}_{ik} e_j$, $c^{j}_{ik} \in F$.
Since  $g[e_i, e_k]=[ge_i, g e_k]$, we have $$\sum_{j=1}^m c^{j}_{ik}\ ge_j
= \sum_{p,q=1}^m [\omega_{pi}(g) e_p, \omega_{qk}(g) e_q],$$
$$\sum_{j,\ell=1}^m c^{j}_{ik} \omega_{\ell j}(g) e_\ell
= \sum_{\ell,p,q=1}^m c^\ell_{pq} \omega_{pi}(g)  \omega_{qk}(g) e_\ell,$$
$$\sum_{j=1}^m c^{j}_{ik} \omega_{\ell j}(g) 
= \sum_{p,q=1}^m c^\ell_{pq} \omega_{pi}(g) \omega_{qk}(g)$$
 for all $g\in G$ and $1 \leqslant i,k,\ell \leqslant m$.
 Hence $$\rho([e_i, e_k])=\sum_{j=1}^m c^{j}_{ik} \rho(e_j)=
 \sum_{j,\ell=1}^m c^{j}_{ik}  e_\ell \otimes \omega_{\ell j}
 = $$ $$\sum_{\ell,p,q=1}^m  c^\ell_{pq} e_\ell \otimes \omega_{pi}\omega_{qk}
 = \sum_{p,q=1}^m  [e_p, e_q] \otimes \omega_{pi}\omega_{qk}
 = [(e_i)_{(0)}, (e_k)_{(0)}] \otimes (e_i)_{(1)}(e_k)_{(1)}.$$
  Thus $L$ is an $\mathcal O(G)$-comodule algebra. The other properties are proved explicitly as well.
\end{proof}

Let $V$ and $W$ be $H$-comodules. We say that a linear map $\varphi \colon V \to W$
is an \textit{$H$-colinear map} if $\varphi(v)_{(0)}\otimes \varphi(v)_{(1)}=\varphi(v_{(0)})
\otimes v_{(1)}$ for all $v \in V$. If $W \subseteq V$ is a subspace and $\rho_V(W)\subseteq W \otimes
H$, we say that $W$ is an \textit{$H$-subcomodule}
or an \textit{$H$-coinvariant subspace}.

\subsection{$H$-module Lie algebras}

Again, let $L$ be a Lie algebra over a field $F$. Suppose $L$ is a left $H$-module where $H$
is a Hopf algebra. We say that $L$ is an
\textit{$H$-module Lie algebra} if $h([a,b])=[h_{(1)}a, h_{(2)}b]$
for all $h \in H$ and $a,b \in L$.

\begin{example} Let $L$ be a Lie $G$-algebra for an arbitrary group $G$, i.e. there is a fixed homomorphism $G \to \Aut(L)$. Then $L$ is an $FG$-module algebra.
\end{example}
\begin{example}\label{ExampleUgModule}
 Let $L$ and $\mathfrak g$ be Lie algebras. Suppose
$\mathfrak g$ is acting on $L$ by derivations, i.e. there is a fixed homomorphism $\mathfrak g \to \Der(L)$
of Lie algebras where $\Der(L)\subseteq \mathfrak{gl}(L)$ is the algebra of derivations.
Consider the corresponding homomorphism $U(\mathfrak g) \to \End_F(L)$ of associative
algebras. Then $L$ becomes an $U(\mathfrak g)$-module algebra.
\end{example}

If $\dim H < +\infty$, then we can define a Hopf algebra structure on the space $H^* := \Hom_F(H, F)$
using the dual operators.
In this case, every $H$-comodule algebra $L$ is an $H^*$-module algebra and vice versa. The correspondence
between the $H$-coaction and the $H^*$-action is given by the formula $h^* a = h^*(a_{(1)})a_{(0)}$
where $h^* \in H^*$, $a \in L$.

\section{Stability of the radicals in Lie algebras}\label{SectionStability}

First we need the following standard

\begin{lemma}\label{LemmaNJadL} Let $L$ be a Lie algebra over 
some field $F$ and let $N$ be a nilpotent ideal of $L$. Denote by $A$ the associative
subalgebra of $\End_F L$ generated by $(\ad L)$. Then $(\ad N) \subseteq J(A)$.
(Here $\ad \colon L \to \mathfrak{gl}(L)$, $(\ad a)b := [a,b]$,
is the adjoint representation of $L$ and $J(A)$ is the Jacobson radical of $A$.)
\end{lemma}
\begin{proof}
Let $N^p=0$, $p \in \mathbb N$.
Then \begin{equation}\label{EqBp0} b_1 \ldots b_p = 0 \text{ for all } b_i \in \ad N.
\end{equation}
Note that if $I$ is the  two-sided ideal of $A$ generated by $(\ad N)$,
then $I^n$, $n\in \mathbb N$, consists of linear combinations of elements
$$(a_{i_{01}} \ldots a_{i_{0s_0}}) b_1 (a_{i_{11}} \ldots a_{i_{1s_1}})
b_2 (a_{i_{21}} \ldots a_{i_{2s_2}})  \ldots b_{n-1}(a_{i_{n-1,1}} \ldots a_{i_{{n-1},s_{n-1}}}) b_n  (a_{i_{n1}} \ldots a_{i_{ns_n}})$$
where $b_i \in \ad N$, $a_{ij} \in \ad L$.
Using $[\ad a, \ad b] = \ad [a, b]$, we move $a_{ij}$ outside and 
obtain that $I^n$ consists of linear combinations of
elements $a(b_1 b_2 \ldots b_n)c$ where $a,c \in A$, $b_i \in \ad N$.
Thus~(\ref{EqBp0}) implies $I^p = 0$ and $(\ad N) \subseteq I \subseteq J(A)$.
\end{proof}

\begin{lemma}\label{LemmaHI}
Let $I$ be an ideal of an $H$-module Lie algebra $L$ where $H$ is a Hopf algebra.
 Then $HI$ is an $H$-invariant ideal of $L$.
\end{lemma}
\begin{proof} Suppose $a \in I$, $h\in H$, $b \in L$. Then
\begin{equation}\label{EqPerebros}
[ha, b] = [h_{(1)}a, \varepsilon(h_{(2)})b]
= [h_{(1)}a, h_{(2)} (Sh_{(3)}) b] = h_{(1)}[a, (Sh_{(2)})b].
\end{equation}
Thus $[ha, b]\in HI$.
\end{proof}

Now we can prove

\begin{theorem}\label{TheoremRadicals}
Let $L$ be a finite dimensional $H$-module Lie algebra over a field $F$ of characteristic $0$
and $H$ be a finite dimensional (co)semisimple Hopf algebra. Then the solvable and
the nilpotent radicals $R$ and $N$ of $L$ are $H$-invariant.
\end{theorem} 

\begin{proof}
We have a natural $H$-action on $\End_F(L)$:
$(h\psi)(a) := h_{(1)} \psi((Sh_{(2)})a)$
for $a \in L$, $h \in H$, $\psi \in \End_F(L)$.
Then $\ad$ becomes a homomorphism of $H$-modules.
Denote by $A$ the associative subalgebra of $\End_F(L)$ generated by $(\ad L)$.
Then $A$ is an $H$-submodule.
By Lemma~\ref{LemmaHI}, $HN$ and $HR$ are ideals of $L$.
Lemma~\ref{LemmaNJadL} implies $(\ad (HN)) \subseteq HJ(A)$.
 By the Larson~--- Radford theorem, $H$  is cosemisimple if and only if it is semisimple (see e.g. \cite[Theorem 7.4.6]{Danara}). Hence,
by~\cite{LinchenkoJH}, $HJ(A)=J(A)$. Thus $HN$ is nilpotent and $HN = N$.

By \cite[Proposition 2.1.7]{GotoGrosshans},
 $[L, R] \subseteq N$. Together with~(\ref{EqPerebros})
 this implies $$[HR, HR] \subseteq [HR, L] \subseteq H [R, HL] \subseteq H[R,L] \subseteq HN = N.$$
 Thus $HR$ is solvable and $HR = R$.
\end{proof}
\begin{corollary}
Let $L$ be a finite dimensional $H$-comodule Lie algebra over a field $F$ of characteristic $0$
and $H$ be a finite dimensional (co)semisimple Hopf algebra. Then the solvable and
the nilpotent radicals $R$ and $N$ of $L$ are $H$-subcomodules in $L$.
\end{corollary}
\begin{proof}
We use the duality between $H$-coactions and $H^*$-actions.
\end{proof}
\begin{corollary}
Let $L$ be a finite dimensional Lie algebra over a field $F$ of characteristic $0$,
graded by a finite group~$G$. Then the solvable and
the nilpotent radicals $R$ and $N$ of $L$ are graded.
\end{corollary}
\begin{proof}
We use Examples~\ref{ExampleFG} and~\ref{ExampleGraded}.
\end{proof}
\begin{remark}
If $F$ is algebraically closed, then $R$ is graded even if $G$ is infinite
(see~\cite[Proposition 3.3 and its proof]{PaReZai}).
\end{remark}
\begin{remark}
Let $L$ and $\mathfrak g$ be Lie algebras over a field of characteristic $0$. Suppose
$\mathfrak g$ is acting on $L$ by derivations and $\dim L < + \infty$.
Then $L$ is an $U(\mathfrak g)$-module algebra (see Example~\ref{ExampleUgModule}).
Note that $U(\mathfrak g)$ is infinite dimensional. Moreover,
by Example~\ref{ExampleUnivEnv}, $U(\mathfrak g)$ does not have nonzero integrals
and, therefore, is not cosemisimple. However, by~\cite[Chapter~III, Section~6, Theorem~7]{JacobsonLie},
$\delta(R) \subseteq N$ for all $\delta \in \Der(L)$. Hence $R$ and $N$ are 
$U(\mathfrak g)$-submodules.
\end{remark}

We conclude the section with an example of an $H$-module Lie algebra
with unstable radicals.
\begin{example}\label{ExampleSweedlerNonStableRadical}
Let $H=\langle 1, g, x, gx \rangle_F$ be the 4-dimensional Sweedler's Hopf algebra
over a field $F$ of characteristic $0$.
Here $g^2=1$, $x^2=0$, $xg=-gx$, $\Delta(g)=g\otimes g$, $\Delta(x)=g\otimes x + x \otimes 1$,
$\varepsilon(g)=1$, $\varepsilon(x)=0$, $S(g)=g$, $S(x)=-gx$. Note that $J(H)=\langle x, gx \rangle_F
\ne 0$, i.e. $H$ is not semisimple.
Let $V$ be a three-dimensional vector space. Fix some linear
isomorphism $\varphi \colon \mathfrak{sl}_2(F) \to V$. Consider
the Lie algebra $L = \mathfrak{sl}_2(F)\oplus V$ 
with the Lie commutator $$[a+\varphi(b), c + \varphi(u)]=[a,c]+\varphi([a,u]+[b,c])
\text{ where } a,b,c,u \in \mathfrak{sl}_2(F),$$
i.e. $V$ is an Abelian ideal of $L$ that coincides with the solvable and the nilpotent radicals
of $L$. Define $g(a+\varphi(b))=a-\varphi(b)$ and $x(a+\varphi(b))=b$
where $a,b \in \mathfrak{sl}_2(F)$.
Then $L$ is an $H$-module Lie algebra, however $V$ is not $H$-invariant.
\end{example}

\section{Levi decompositions for $H$-comodule Lie algebras}\label{SectionHcoLevi}

First we recall the basic concepts of Lie algebra cohomology. For the details we refer the reader
to e.g.~\cite{GotoGrosshans, Postnikov}.

Let  $\psi \colon L \to \mathfrak{gl}(V)$ be a representation
of a Lie algebra $L$ on some vector space $V$ over a field $F$.
Denote by $C^k(L; V) \subseteq \Hom_F(L^{{}\otimes k}; V)$, $k \in \mathbb N$, the subspace
of all alternating multilinear maps, $C^0(L; V) := V$.
Recall that the elements of $C^k(L; V)$ are called \textit{$k$-cochains}
with coefficients in $V$.  The \textit{coboundary operators} $\partial \colon C^k(L;V) \to C^{k+1}(L; V)$
are defined on these spaces in such a way that $\partial^2 = 0$.
The elements of the subspace $$Z^k(L;\psi):= \ker(\partial \colon C^k(L;V) \to C^{k+1}(L; V))\subseteq C^k(L;V)$$
are called \textit{$k$-cocycles} and the elements of the subspace $$B^k(L;\psi):= \partial(C^{k-1}(L;V)) \subseteq C^k(L;V)$$
are called \textit{$k$-coboundaries}. The space $H^k(L;\psi) := Z^k(L;\psi)/B^k(L;\psi)$
is called the \textit{$k$th cohomology group}.

Suppose $L$ is an $H$-comodule algebra for some Hopf algebra $H$.
 Consider a representation $\psi \colon L \to \mathfrak{gl}(V)$.
 We say that $(V, \psi)$ is an \textit{$(H,L)$-module} if $V$ is an $H$-comodule
 and $$\rho_V(\psi(a)v)=\psi(a_{(0)}) v_{(0)}\otimes a_{(1)}v_{(1)} \text{ for all } a \in L,\ v \in V$$
 where $\rho_V \colon V \to V \otimes H$ is the comodule map.
 We say that $(V, \psi)$ is a \textit{symmetric $(H,L)$-module} 
 if $$\rho_V(\psi(a)v)=\psi(a_{(0)}) v_{(0)}\otimes a_{(1)}v_{(1)} =
 \psi(a_{(0)}) v_{(0)}\otimes v_{(1)}a_{(1)} \text{ for all } a \in L,\ v \in V.$$
\begin{example}
If $L$ is an $H$-comodule Lie algebra, then the adjoint representation
$\ad \colon L \to \mathfrak{gl}(L)$ defines on $L$
the structure of a symmetric $(H,L)$-module
since $$\rho((\ad a)b)=\rho([a,b])= -\rho([b,a]) =-[b_{(0)}, a_{(0)}]\otimes b_{(1)}a_{(1)} =
 (\ad a_{(0)}) b_{(0)}\otimes b_{(1)}a_{(1)}$$ for all $a,b \in L$.
\end{example}

Denote by $\tilde C^k(L; V)$ the subspace of \textit{$H$-colinear cochains}, i.e. such maps
 $\omega \in C^k(L; V)$ that $$\rho_V(\omega(a_1, a_2, \ldots, a_k))
 =\omega({a_1}_{(0)}, {a_2}_{(0)}, \ldots, {a_k}_{(0)})\otimes
 {a_1}_{(1)} {a_2}_{(1)} \ldots {a_k}_{(1)} \text{ for all }a_i \in L.$$
  
 If $(V,\psi)$ is an $(H,L)$-module and $H$ is commutative, then, clearly, the coboundary of an
 $H$-colinear cochain is again an $H$-colinear cochain. However, for $1$-cochains and a symmetric
 $(H,L)$-module $(V,\psi)$ this is true even if $H$ is not commutative.

\begin{lemma}\label{LemmaHCoboundary}
If $(V,\psi)$ is a symmetric $(H,L)$-module, then $\partial(\tilde C^1(L;V)) 
\subseteq \tilde C^2(L;V)$.
\end{lemma}
\begin{proof}
Let $\omega \in \tilde C^1(L;V)$.
Then $$(\partial\omega)(x,y)
:=\psi(x)\omega(y)-\psi(y)\omega(x) - \omega([x,y])$$
and $$\rho_V((\partial\omega)(x,y))= $$ $$\psi(x_{(0)})\omega(y)_{(0)}
\otimes x_{(1)}\omega(y)_{(1)} -\psi(y_{(0)})\omega(x)_{(0)}
\otimes {\omega(x)_{(1)}} y_{(1)} - \omega([x,y]_{(0)}) \otimes [x,y]_{(1)}
=$$ $$ \psi(x_{(0)})\omega(y_{(0)})
\otimes x_{(1)}y_{(1)} -\psi(y_{(0)})\omega(x_{(0)})
\otimes x_{(1)}y_{(1)} - \omega([x_{(0)},y_{(0)}]) \otimes x_{(1)}y_{(1)}
=$$ $$(\partial\omega)(x_{(0)},y_{(0)}) \otimes x_{(1)}y_{(1)}.$$
\end{proof}

Let $\tilde Z^2(L; \psi) := Z^2(L; \psi) \cap \tilde C^2(L;V)$
and $\tilde B^2(L; \psi) := \partial(\tilde C^1(L;V))$.
Lemma~\ref{LemmaHCoboundary} enables us to define
\textit{the second $H$-colinear cohomology group} $\tilde H^2(L; \psi) :=
 \tilde Z^2(L; \psi)/\tilde B^2(L; \psi)$.
 
In~\cite{Taft} E.\,J.~Taft used the original Maschke
trick in order to replace non-invariant maps with invariant ones.
In~\cite{SteVanOyst} D. \c Stefan and F. Van Oystaeyen used Maschke's
trick adapted for Hopf algebras with a left integral.
We use the last one too.
 
 \begin{lemma}\label{LemmaHcolinear}
 Let $r \colon V \to W$ be a linear map where $V$ and $W$ are $H$-comodules for a Hopf algebra $H$. Let $t \in H^*$ be a left integral on $H$. Then $\tilde r \colon V \to W$ where
 $$\tilde r(x) = t\bigl(r(x_{(0)})_{(1)}S(x_{(1)})\bigr)r(x_{(0)})_{(0)} \text{ for } x \in V,$$ is an $H$-colinear map. If, in addition, $\pi \circ r = \id_V$ for some
 $H$-colinear map $\pi \colon W \to V$ and $t(1)=1$, then $\pi \circ \tilde r = \id_V$ too.
 \end{lemma}
 \begin{remark} If $G$ is an arbitrary group and $H=FG$, then $V=\bigoplus_{g \in G}V^{(g)}$ and $W=\bigoplus_{g \in G}W^{(g)}$ are graded spaces.
 Suppose $t$ is from Example~\ref{ExampleFG}.
 Then $\tilde r(x) = \sum_{g \in G} p_{W,g}\ r(p_{V,g} x)$ for $x \in V$
 and this is a graded map.
 Here $p_{V,g}$ is the projection of $V$ on $V^{(g)}$
 along $\bigoplus_{\substack{h \in G, \\ h \ne g}} V^{(h)}$ and
 $p_{W,g}$ is the projection of $W$ on $W^{(g)}$
 along $\bigoplus_{\substack{h \in G, \\ h \ne g}} W^{(h)}$.
 \end{remark}
 \begin{proof}[Proof of Lemma~\ref{LemmaHcolinear}]
 Note that $$
 \tilde r(x_{(0)}) \otimes x_{(1)}=
 t\bigl(r(x_{(0)})_{(1)}S(x_{(1)})\bigr)\ r(x_{(0)})_{(0)} \otimes x_{(2)}
 = $$
 $$r(x_{(0)})_{(0)} \otimes \Bigl(t\bigl(r(x_{(0)})_{(1)}S(x_{(1)})\bigr)1\Bigr)x_{(2)}=
 $$
 $$
 r(x_{(0)})_{(0)} \otimes t\bigl(r(x_{(0)})_{(1)(2)}(Sx_{(1)})_{(2)}\bigr)\ r(x_{(0)})_{(1)(1)}(Sx_{(1)})_{(1)}
  x_{(2)} =$$
  $$ r(x_{(0)})_{(0)} \otimes t\bigl(r(x_{(0)})_{(2)}S(x_{(1)})\bigr)\ r(x_{(0)})_{(1)}(Sx_{(2)})x_{(3)}
  = $$ 
  $$t\bigl(r(x_{(0)})_{(2)}S(x_{(1)})\bigr)\ r(x_{(0)})_{(0)} \otimes r(x_{(0)})_{(1)}
  = \rho_W(\tilde r(x)).$$
 Thus $\tilde r$ is $H$-colinear and the first part of the lemma is proved.
 
 Suppose $\pi \circ r = \id_V$ for some
 $H$-colinear map $\pi \colon W \to V$. Let $x \in V$.
 Then $$(\pi \circ \tilde r)(x)=
 t\bigl(r(x_{(0)})_{(1)}S(x_{(1)})\bigr)\ \pi\bigl(r(x_{(0)})_{(0)}\bigr)
 = t\bigl((\pi\circ r)(x_{(0)})_{(1)}S(x_{(1)})\bigr)\ (\pi\circ r)(x_{(0)})_{(0)}
 =$$ $$ t\bigl(x_{(0)(1)}S(x_{(1)})\bigr)\ x_{(0)(0)} =
 t\bigl(x_{(1)}S(x_{(2)})\bigr)\ x_{(0)} = t(1) x = x.$$
 \end{proof}

\begin{lemma}\label{LemmaH2zero}
Let $(V,\psi)$ be a finite dimensional symmetric $(H,L)$-module
where $L$ is a finite dimensional $H$-comodule semisimple Lie algebra over a field $F$ of characteristic $0$ and $H$ is a Hopf algebra with an $\ad$-invariant left integral $t \in H^*$, $t(1)=1$.
 Then $\tilde H^2(L; \psi) = 0$.
\end{lemma}
\begin{proof}
Recall that by the second Whitehead lemma (see e.g. \cite[Exercise~3.5]{GotoGrosshans}
or \cite[Lecture 19, Corollary 2]{Postnikov}), $H^2(L; \psi) = 0$.
Hence if $\omega \in \tilde Z^2(L; \psi)$, then there exists $\nu \in C^1(L; V)$
such that $\omega = \partial \nu$. Let $\tilde \nu$ be the map obtained from $\nu$
in Lemma~\ref{LemmaHcolinear}. 
Then $\tilde \nu \in \tilde C^1(L; V)$. We claim that $\partial\tilde\nu = \omega$.

Let $a,b \in L$. We have
$$(\partial\tilde\nu)(a, b) = \psi(a)\tilde\nu(b)
- \psi(b)\tilde\nu(a)-\tilde\nu([a,b])= t(\nu(b_{(0)})_{(1)} S(b_{(1)}))\ \psi(a) \nu(b_{(0)})_{(0)}
- $$
$$t(\nu(a_{(0)})_{(1)} S(a_{(1)}))\ \psi(b) \nu(a_{(0)})_{(0)}
-t(\nu([a,b]_{(0)})_{(1)} S([a,b]_{(1)}))\ \nu([a,b]_{(0)})_{(0)}=$$
$$ t(\nu(b_{(0)})_{(1)} S(b_{(1)}))\ \psi(\varepsilon(a_{(1)})a_{(0)}) \nu(b_{(0)})_{(0)}
- $$
$$t(\nu(a_{(0)})_{(1)} S(a_{(1)}))\ \psi(\varepsilon(b_{(1)})b_{(0)}) \nu(a_{(0)})_{(0)}
-t(\nu([a_{(0)},b_{(0)}])_{(1)} S(a_{(1)}b_{(1)}))\ \nu([a_{(0)},b_{(0)}])_{(0)}=$$
$$ \varepsilon(a_{(1)})t\Bigl( \nu(b_{(0)})_{(1)}S(b_{(1)})\Bigr)\ \psi(a_{(0)}) \nu(b_{(0)})_{(0)}
- $$
$$t\Bigl(\nu(a_{(0)})_{(1)}\varepsilon(b_{(1)})S(a_{(1)})\Bigr)\ \psi(b_{(0)}) \nu(a_{(0)})_{(0)}
-$$
$$t\Bigl(\nu([a_{(0)},b_{(0)}])_{(1)}S(a_{(1)}b_{(1)})\Bigr)\ \nu([a_{(0)},b_{(0)}])_{(0)}.$$

Since $t$ is $\ad$-invariant, we get
$$(\partial\tilde\nu)(a, b) =  t\Bigl(a_{(1)} \nu(b_{(0)})_{(1)}(Sb_{(1)})S(a_{(2)})\Bigr)\ \psi(a_{(0)}) \nu(b_{(0)})_{(0)}
- $$
$$t\Bigl(\nu(a_{(0)})_{(1)}b_{(1)}(Sb_{(2)})S(a_{(1)})\Bigr)\ \psi(b_{(0)}) \nu(a_{(0)})_{(0)}
-$$
$$t\Bigl(\nu([a_{(0)},b_{(0)}])_{(1)}S(a_{(1)}b_{(1)})\Bigr)\ \nu([a_{(0)},b_{(0)}])_{(0)}.$$
Since $(V, \psi)$ is a symmetric $(H,L)$-module, we have
$$(\partial\tilde\nu)(a, b) = t\Bigl(\bigl(a_{(0)(1)} \nu(b_{(0)})_{(1)}\bigr)(Sb_{(1)})S(a_{(1)})\Bigr)\ \psi(a_{(0)(0)}) \nu(b_{(0)})_{(0)}
- $$
$$t\Bigl(\bigl(\nu(a_{(0)})_{(1)}b_{(0)(1)}\bigr)(Sb_{(1)})S(a_{(1)})\Bigr)\ \psi(b_{(0)(0)}) \nu(a_{(0)})_{(0)}
-$$
$$t\Bigl(\nu([a_{(0)},b_{(0)}])_{(1)}S(a_{(1)}b_{(1)})\Bigr)\ \nu([a_{(0)},b_{(0)}])_{(0)}=$$
$$ t\Bigl(\bigl(\psi(a_{(0)}) \nu(b_{(0)})\bigr)_{(1)}(Sb_{(1)})S(a_{(1)})\Bigr)\ \bigl(\psi(a_{(0)}) \nu(b_{(0)})\bigr)_{(0)}
- $$
$$t\Bigl(\bigl(\psi(b_{(0)}) \nu(a_{(0)})\bigr)_{(1)}(Sb_{(1)})S(a_{(1)})\Bigr)\ \bigl(\psi(b_{(0)}) \nu(a_{(0)})\bigr)_{(0)}
-$$
$$t\Bigl(\nu([a_{(0)},b_{(0)}])_{(1)}S(a_{(1)}b_{(1)})\Bigr)\ \nu([a_{(0)},b_{(0)}])_{(0)}=$$
$$ t\Bigl(\bigl(\psi(a_{(0)}) \nu(b_{(0)})\bigr)_{(1)}S(a_{(1)}b_{(1)})\Bigr)\ \bigl(\psi(a_{(0)}) \nu(b_{(0)})\bigr)_{(0)}
- $$
$$t\Bigl(\bigl(\psi(b_{(0)}) \nu(a_{(0)})\bigr)_{(1)}S(a_{(1)}b_{(1)})\Bigr)\ \bigl(\psi(b_{(0)}) \nu(a_{(0)})\bigr)_{(0)}
-$$
$$t\Bigl(\nu([a_{(0)},b_{(0)}])_{(1)} S(a_{(1)}b_{(1)})\Bigr)\ \nu([a_{(0)},b_{(0)}])_{(0)}=$$
 $$
t\bigl(\omega(a_{(0)}, b_{(0)})_{(1)}S(a_{(1)}b_{(1)})\bigr)\ \omega(a_{(0)}, b_{(0)})_{(0)}
=t\bigl(a_{(1)}b_{(1)} S(a_{(2)}b_{(2)})\bigr)\ \omega(a_{(0)}, b_{(0)})=\omega(a,b)$$ since $\omega \in \tilde Z^2(L; \psi)$ and $t(1)=1$. Thus 
$\tilde Z^2(L; \psi) = \tilde B^2(L; \psi)$ and $\tilde H^2(L; \psi)=0$.
\end{proof}

Theorem~\ref{TheoremHcoLevi} is the $H$-comodule version of the Levi theorem.

\begin{theorem}\label{TheoremHcoLevi}
Let $L$ be a finite dimensional $H$-comodule Lie algebra over a field $F$ of characteristic $0$
where $H$ is a Hopf algebra.  Suppose $R$ is an $H$-subcomodule and 
there exists an $\ad$-invariant left integral $t \in H^*$ such that $t(1)=1$.
Then there exists a maximal semisimple subalgebra $B$ in $L$ such that
$L=B\oplus R$ (direct sum of $H$-subcomodules).
\end{theorem}

\begin{proof}
We follow the outline of the proof of the original Levi theorem.

First we suppose that the solvable radical $R$ is Abelian. Let $\pi \colon L \to L/R$ be the natural projection. Note that $L/R$ is an $H$-comodule semisimple Lie algebra and
 $\pi$ is an $H$-colinear map since $R$ is an $H$-subcomodule.  Take any linear section $r \colon L/R \to L$ such that $\pi \circ r = \id_{L/R}$.
By Lemma~\ref{LemmaHcolinear}, we may assume that $r$ is $H$-colinear.

Denote $\Phi(x,y)=[r(x),r(y)]-r([x,y])$. Note that $\pi(\Phi(a,b))=[a,b]-[a,b]=0$
for all $a,b \in L/R$. Thus $\Phi \in \tilde C^2(L/R; R)$.
Let $\psi \colon L/R \to \mathfrak{gl}(R)$ be the linear map defined
by $\psi(a)(v)=[r(a), v]$, $a\in L/R$, $v\in R$. Note that
 $$\psi(a)\psi(b)(v)-\psi(b)\psi(a)(v)-\psi([a,b])(v)=[r(a), [r(b), v]]-[r(b), [r(a), v]]
 -[r([a, b]), v]=$$ $$[[r(a),r(b)],v]-[r([a, b]), v]=[\Phi(a,b), v]=0
 \text{ for all } a,b \in L/R$$ since $\Phi(a,b) \in R$ and $[R,R]=0$.
 Thus $\psi$ is a representation of $L/R$. Moreover, $(R, \psi)$
 is a symmetric $(H, L/R)$-module. Note that
  $\Phi \in \tilde Z^2(L/R; \psi)$  since $\partial \Phi = 0$.
 Hence, by Lemma~\ref{LemmaH2zero},
 $\Phi=\partial \omega$ for some $\omega \in \tilde C^1(L/R; R)$.
 Then $$[(r-\omega)(a),(r-\omega)(b)]-(r-\omega)([a,b])=$$ $$
 ([r(a),r(b)]-r([a,b]))-([r(a),\omega(b)]-[r(b),\omega(a)]-\omega([a,b]))+[\omega(a),\omega(b)]=$$ $$
 \Phi(a,b)-(\partial\omega)(a,b)+0=0$$  for all $a,b \in L/R$ and $\pi \circ (r-\omega) = \pi \circ r = \id_{L/R}$. Therefore, $(r-\omega)$ is an $H$-colinear homomorphic embedding of $L/R$ into $L$ and $L=B\oplus R$ (direct sum of $H$-subcomodules) for $B=(r-\omega)(L/R)$, i.e in this case the $H$-coinvariant Levi theorem is proved.
 
 Now prove the general case by induction on $\dim R$. The theorem has been already proved
 for the case when $[R,R]=0$. Suppose $[R,R]\ne 0$. Note that $[R,R]\ne R$ since $R$ is solvable.
 Moreover $[R,R]$ is an $H$-subcomodule. Now we consider $L/[R,R]$. Since $(L/[R,R])/(R/[R,R])\cong L/R$ is semisimple,
 $R/[R,R]$ is the solvable radical of $L/[R,R]$. We apply the induction hypothesis and
 obtain that $L/[R,R] = L_1/[R,R] \oplus R/[R,R]$ (direct sum of $H$-subcomodules)
 for some $H$-coinvariant subalgebra $L_1 \subset L$ where
 $L_1/[R,R] \cong L/R$. Now we apply the induction hypothesis to
 $L_1$ and get $L_1 = B \oplus [R,R]$ (direct sum of $H$-subcomodules)
 where $B \cong L_1/[R,R] \cong L/R$ is a semisimple subalgebra.
 Hence $L = B \oplus R$ (direct sum of $H$-subcomodules) and the theorem is proved.
 \end{proof}
 
 \section{Levi decompositions for graded and $H$-module Lie algebras}\label{SectionGradHLevi}

  Here we obtain some important consequences of Theorem~\ref{TheoremHcoLevi}.
 
\begin{theorem}\label{TheoremHLevi}
Let $L$ be a finite dimensional $H$-module Lie algebra over a field $F$ of characteristic $0$
where $H$ is a finite dimensional (co)semisimple Hopf algebra. 
Then there exists a maximal semisimple subalgebra $B$ in $L$ such that
$L=B\oplus R$ (direct sum of $H$-submodules).
\end{theorem}
\begin{proof}
Using the duality and combining Example~\ref{ExampleHSS}, Theorems~\ref{TheoremRadicals} and \ref{TheoremHcoLevi}, we obtain the theorem.
\end{proof}

D.~Pagon, D.~Repov\v s, and M.V.~Zaicev~\cite{PaReZai}
proved the graded version of the Levi theorem for a finite group.
Now we can show that this is true for an arbitrary group.

\begin{theorem}\label{TheoremGradLevi}
Let $L$ be a finite dimensional Lie algebra over a field $F$ of characteristic $0$,
graded by an arbitrary group $G$. Suppose that the solvable radical $R$ of $L$ is graded.
Then there exists a maximal semisimple subalgebra $B$ in $L$ such that
$L=B\oplus R$ (direct sum of graded subspaces).
\end{theorem}
\begin{proof}
We use Examples~\ref{ExampleFG}, \ref{ExampleGraded}, and Theorem~\ref{TheoremHcoLevi}.
\end{proof}

By~\cite[Proposition 3.3 and its proof]{PaReZai}, if $F$ is algebraically closed, then $R$ is graded.
Thus we obtain
\begin{corollary}
Let $L$ be a finite dimensional Lie algebra over an algebraically closed field $F$ of characteristic $0$,
graded by an arbitrary group $G$. Then there exists a maximal semisimple subalgebra $B$ in $L$ such that
$L=B\oplus R$ (direct sum of graded subspaces).
\end{corollary}

Now we apply Theorem~\ref{TheoremHcoLevi} to Lie algebras with a rational action
of an affine algebraic group.

\begin{theorem}\label{TheoremAffAlgGrLevi}
Let $L$ be a finite dimensional Lie algebra over an algebraically closed field $F$ of characteristic $0$
and let $G$ be a reductive affine algebraic group over $F$.
Suppose $L$ is endowed with a rational action of $G$ by automorphisms.
 Then there exists a maximal semisimple subalgebra $B$ in $L$ such that
$L=B\oplus R$ (direct sum of $G$-invariant subspaces).
\end{theorem}
\begin{proof}
We notice that $R$ is invariant under all automorphisms and 
use Examples~\ref{ExampleAffAlgGr}, \ref{ExampleRatAffAlgGrAction}, and Theorem~\ref{TheoremHcoLevi}.
\end{proof}

We conclude the section with examples of $H$-module Lie algebras for which an $H$-invariant
Levi decomposition does not exist.

\begin{example}[Yuri Bahturin]\label{ExampleHnoninvLevi}
Let $L = \left\lbrace\left(\begin{array}{cc} C & D \\
0 & 0
  \end{array}\right) \mathrel{\biggl|} C \in \mathfrak{sl}_m(F), D\in M_m(F)\right\rbrace
  \subseteq \mathfrak{sl}_{2m}(F)$, $m \geqslant 2$.
  Then $$R=\left\lbrace\left(\begin{array}{cc} 0 & D \\
0 & 0
  \end{array}\right) \mathrel{\biggl|} D\in M_m(F)\right\rbrace
  $$
  is the solvable (and nilpotent) radical of $L$.
    Consider $\varphi \in \Aut(L)$ where
  $$\varphi\left(\begin{array}{cc} C & D \\
0 & 0
  \end{array}\right)=\left(\begin{array}{cc} C & C+D \\
0 & 0
  \end{array}\right).$$
  Then $L$ is a $G$-algebra and an $FG$-module algebra where $G=\langle \varphi \rangle
  \cong \mathbb Z$. However there is no $FG$-invariant semisimple subalgebra $B$
  such that $L=B\oplus R$ (direct sum of $FG$-submodules).
\end{example}
\begin{proof} 
Let $a \in L$. Then $\varphi(a)-a \in R$. Suppose $B$ is a $G$-invariant subspace
such that $B \cap R = \lbrace 0 \rbrace$. Then $\varphi(b)-b=0$ for all $b \in B$
and $B \subseteq R$. Hence $B=0$ and there is no $H$-invariant Levi decomposition.
\end{proof}
\begin{example}
Let $L$ be the same algebra as in Example~\ref{ExampleHnoninvLevi}
with an adjoint action on itself by derivations. Then $L$ becomes an $U(L)$-module Lie algebra
(see Example~\ref{ExampleUgModule}),
and $U(L)$-submodules of $L$ are exactly the ideals of $L$.
However there is no $U(L)$-invariant semisimple subalgebra $B$
  such that $L=B\oplus R$ (direct sum of $U(L)$-submodules).
\end{example}
\begin{proof}
Suppose $L=B\oplus R$ where $B$ is an $U(L)$-submodule. Then $B$ is an ideal of $L$
and $R$ is the center of $L$ since $[R,R]=0$. We get a contradiction. Thus there is
no $U(L)$-invariant Levi decomposition.
\end{proof}

\section{$H$-(co)invariant decompositions of semisimple algebras}\label{SectionHSS}

First we need two auxiliary propositions.

\begin{lemma}\label{LemmaHcentral}
Let $L$ be an $H$-module Lie algebra for some Hopf algebra $H$ and $M$
be an $H$-submodule of $L$. Denote by $N$ the centralizer of $M$ in $L$,
i.e. $N = \lbrace a \in L \mid [a,b]=0 \text{ for all } b\in M \rbrace$.
Then $N$ is an $H$-submodule of $L$ too.
\end{lemma}
\begin{proof}
Let $a \in N$,
 $b \in M$, $h \in H$. Then $$[ha, b]=[\varepsilon(h_{(2)})h_{(1)}a, b]=
 [h_{(1)}a, h_{(2)}(Sh_{(3)})b]=h_{(1)}[a, (Sh_{(2)})b]=0$$ since for each summand
 $(Sh_{(2)})b \in M$. Hence $N$ is an $H$-submodule.
 \end{proof}
 
 The dual result is
\begin{lemma}\label{LemmacoHcentral}
Let $L$ be an $H$-comodule Lie algebra for some Hopf algebra $H$ and $M$
be an $H$-subcomodule of $L$. Denote by $N$ the centralizer of $M$ in $L$.
Then $N$ is an $H$-subcomodule of $L$ too.
\end{lemma}
\begin{proof}
 It is sufficient to prove that for every $h^* \in H^*$, $a \in N$,
 $b \in M$,
 we have $[h^*(a_{(1)})a_{(0)}, b]=0$. Note that
   $$[h^*(a_{(1)})a_{(0)}, b]=h^*(a_{(1)})[a_{(0)}, b]=h^*(a_{(1)})[a_{(0)}, \varepsilon(b_{(1)}) b_{(0)}]=
  h^*(a_{(1)}\varepsilon(b_{(1)})1)[a_{(0)}, b_{(0)}]=$$ $$
  h^*(a_{(1)}b_{(1)} S(b_{(2)}))[a_{(0)}, b_{(0)}]=
  h^*([a,b_{(0)}]_{(1)} S(b_{(1)}))[a,b_{(0)}]_{(0)}
   =0$$ since for each summand
 $[a,b_{(0)}]=0$. Hence $N$ is an $H$-subcomodule.
 \end{proof}

Let $H$ be a Hopf algebra and let $B$ be an $H$-module Lie algebra.
We say that $B$ is \textit{$H$-simple} if for every ideal $I$ of $B$ such that $I$ is an $H$-submodule,
either $I=0$ or $B=I$.

 Analogously, if $B$ is an $H$-comodule Lie algebra, then $B$ is \textit{$H$-simple} if for every ideal $I$ of $B$ such that $I$ is an $H$-subcomodule,
either $I=0$ or $B=I$.

By a classical theorem of Lie theory (see e.g. \cite[Theorem 2.1.4]{GotoGrosshans}), every finite dimensional semisimple Lie algebra is a direct sum of simple Lie algebras. Here we prove the corresponding
result for $H$-(co)module Lie algebras. Our proof is based on~\cite[Proposition 3.1]{PaReZai}.

\begin{theorem}\label{TheoremHsemisimple}
Let $B$ be a finite dimensional semisimple $H$-module Lie algebra over a field of characteristic $0$
where $H$ is an arbitrary Hopf algebra.
Then $B=B_1 \oplus B_2 \oplus \ldots \oplus B_s$ (direct sum of ideals and $H$-submodules) for
some $H$-simple subalgebras $B_i$.
\end{theorem} 
\begin{proof} We prove the theorem by induction on $\dim B$. If $B$ is $H$-simple,
then there is nothing to prove.

Let $\tilde B \subset B$ be some minimal $H$-invariant ideal of 
$B$ and  let $B=\tilde B_1 \oplus \tilde B_2 \oplus \ldots \oplus \tilde B_q$
where $\tilde B_j$ are simple ideals of $B$. Then $\tilde B = \tilde B_{i_1} \oplus \tilde B_{i_2}
 \oplus \ldots \oplus \tilde B_{i_\ell}$ for some $i_k$ and
 the centralizer $B_0$ of $\tilde B$ in $B$ consists of the direct sum of the rest $\tilde B_j$.
  By Lemma~\ref{LemmaHcentral}, $B_0$ is an $H$-invariant ideal,
   $B=\tilde B \oplus B_0$, and we can apply the induction.
\end{proof}

The dual result is
\begin{theorem}\label{TheoremcoHsemisimple}
Let $B$ be a finite dimensional semisimple $H$-comodule Lie algebra over a field of characteristic $0$
where $H$ is an arbitrary Hopf algebra.
Then $B=B_1 \oplus B_2 \oplus \ldots \oplus B_s$ (direct sum of ideals and $H$-subcomodules) for
some $H$-simple subalgebras $B_i$.
\end{theorem}
\begin{proof}
We repeat verbatim the proof of Theorem~\ref{TheoremHsemisimple} using Lemma~\ref{LemmacoHcentral}
instead of Lemma~\ref{LemmaHcentral}.
\end{proof}

\section{$H$-(co)invariant analog of the Weyl theorem}\label{SectionHWeyl}

First, we need the following addition to Lemma~\ref{LemmaHcolinear}.
\begin{lemma}\label{LemmaHLlinear}
Let $\pi \colon V \to W$ be a homomorphism of $L$-modules where $(V, \varphi)$ and $(W, \psi)$ are $(H,L)$-modules for an $H$-comodule Lie algebra $L$ and a Hopf algebra $H$. Let $t \in H^*$ be an $\ad$-invariant
 left integral on $H$. Then $\tilde \pi \colon V \to W$ where
 $$\tilde \pi(x) = t\bigl(\pi(x_{(0)})_{(1)}S(x_{(1)})\bigr)\pi(x_{(0)})_{(0)} \text{ for } x \in V,$$ is an
 $H$-colinear homomorphism of $L$-modules. Moreover, if $t(1)=1$, $W \subseteq V$, 
 and $\pi$ is a projection of $V$ on $W$, then $\tilde\pi$ is a projection of $V$ on $W$ too.
 
\end{lemma}
\begin{proof} The map $\tilde \pi$ is $H$-colinear by Lemma~\ref{LemmaHcolinear}. 
Let $a \in L$, $x\in V$. Then $$\tilde \pi(\varphi(a)x)= 
t\bigl(\pi((\varphi(a)x)_{(0)})_{(1)}S((\varphi(a)x)_{(1)})\bigr)\pi((\varphi(a)x)_{(0)})_{(0)}=$$
$$t\bigl(\pi(\varphi(a_{(0)})x_{(0)})_{(1)}S(a_{(1)}x_{(1)})\bigr)\pi(\varphi(a_{(0)})x_{(0)})_{(0)}=$$
$$t\bigl((\psi(a_{(0)})\pi(x_{(0)}))_{(1)}S(a_{(1)}x_{(1)})\bigr)(\psi(a_{(0)})\pi(x_{(0)}))_{(0)}=$$
$$t\bigl(a_{(1)}\pi(x_{(0)})_{(1)}S(a_{(2)}x_{(1)})\bigr)\psi(a_{(0)})\pi(x_{(0)})_{(0)}=$$
$$t\bigl(a_{(1)}\pi(x_{(0)})_{(1)}(Sx_{(1)})S(a_{(2)})\bigr)\psi(a_{(0)})\pi(x_{(0)})_{(0)}=$$
$$t\bigl(\pi(x_{(0)})_{(1)}S(x_{(1)})\bigr)\psi(a)\pi(x_{(0)})_{(0)}
=\psi(a)\tilde \pi(x)$$
since $t$ is $\ad$-invariant. Hence $\tilde\pi$ is an
 $H$-colinear homomorphism of $L$-modules.
 
 Let $t(1)=1$, $W \subseteq V$,
 and $\pi$ is a projection of $V$ on $W$. Consider $x \in W$.
 Since $W$ is an $H$-subcomodule, $x_{(0)}\otimes x_{(1)} \in W \otimes H$.
 Thus $$\tilde \pi(x) = t\bigl(\pi(x_{(0)})_{(1)}S(x_{(1)})\bigr)\pi(x_{(0)})_{(0)}
 =$$ $$t\bigl(x_{(0)(1)}S(x_{(1)})\bigr)x_{(0)(0)}
 =t\bigl(x_{(1)}S(x_{(2)})\bigr)x_{(0)}
 =t(1)x=x$$ and $\tilde\pi$ is a projection of $V$ on $W$ too.
\end{proof}

We say that an $(H,L)$-module $(V, \psi)$ is \textit{irreducible}
if it has no nontrivial $L$-submodules that are $H$-subcomodules at the same time.

\begin{theorem}\label{TheoremHcoWeyl}
Let $L$ be an $H$-comodule Lie algebra over a field of characteristic $0$,
let $H$ be a Hopf algebra with an $\ad$-invariant integral $t \in H^*$, $t(1)=1$,
and let $(V, \psi)$ be a finite dimensional $(H,L)$-module
completely reducible as an $L$-module disregarding the $H$-coaction. Then $V=V_1 \oplus V_2 \oplus \ldots \oplus V_s$ for some irreducible $(H,L)$-submodules $V_i$. 
\end{theorem}
\begin{proof}
Here we again use Maschke's trick. It is sufficient to prove that for every $H$-coinvariant
$L$-submodule $W \subseteq V$ there exists a projection $\tilde \pi \colon V \to W$ that is an $H$-colinear
homomorphism of $L$-modules. Then $W = V \oplus \ker \tilde\pi$ and we can use the induction on $\dim V$.

Since $V$ is a completely reducible $L$-module, there exists
a projection $\pi \colon V \to W$ that is a homomorphism of $L$-modules. Now we take the projection $\tilde \pi$
defined in Lemma~\ref{LemmaHLlinear}, which is an $H$-colinear
homomorphism of $L$-modules.
\end{proof}
Now we prove the analog of Weyl
theorem~\cite[Theorem 6.3]{Humphreys} for $H$-comodule Lie algebras.
\begin{corollary}
Let $L$ be a finite dimensional semisimple $H$-comodule Lie algebra over a field of characteristic $0$,
let
$H$ be a Hopf algebra with an $\ad$-invariant integral $t \in H^*$, $t(1)=1$,
and let $(V, \psi)$ be a finite dimensional $(H,L)$-module. Then $V=V_1 \oplus V_2 \oplus \ldots \oplus V_s$
for some irreducible $(H,L)$-submodules $V_i$.
\end{corollary}
\begin{proof}
By the original Weyl theorem, $V$ is a completely reducible $L$-module.
Now we apply Theorem~\ref{TheoremHcoWeyl}.
\end{proof}

Again, we have the variants of the theorem for $H=FG$ and for finite dimensional $H$.

Let $G$ be a group, let $L=\bigoplus_{g\in G} L^{(g)}$ be a graded Lie algebra, and let $V=\bigoplus_{g\in G} V^{(g)}$ be a $G$-graded vector space.
 We say that $(V, \psi)$, where $\psi \colon L \to \mathfrak{gl}(V)$,
 is a \textit{graded $L$-module} if $\psi(a^{(g)})v^{(h)} \in V^{(gh)}$
 for all $g,h \in G$, $a^{(g)} \in L^{(g)}$, $v^{(h)} \in V^{(h)}$.
  We say that an graded $L$-module $(V, \psi)$ is \textit{irreducible}
if it has no nontrivial graded $L$-submodules.

\begin{theorem}\label{TheoremGradWeyl}
Let $L$ be a Lie algebra over a field of characteristic $0$
graded by any group,
and let $(V, \psi)$ be a finite dimensional graded $L$-module completely reducible as an $L$-module
disregarding the grading.
 Then $$V=V_1 \oplus V_2 \oplus \ldots \oplus V_s$$
for some irreducible graded $L$-submodules $V_i$. 
\end{theorem} 
\begin{proof}
We use Examples~\ref{ExampleFG}, \ref{ExampleGraded}, and Theorem~\ref{TheoremHcoWeyl}.
\end{proof}
\begin{corollary}
Let $L$ be a finite dimensional semisimple Lie algebra over a field of characteristic $0$
graded by any group,
and let $(V, \psi)$ be a finite dimensional graded $L$-module.
 Then $V=V_1 \oplus V_2 \oplus \ldots \oplus V_s$
for some irreducible graded $L$-submodules $V_i$. 
\end{corollary}

Let $L$ be an $H$-module Lie algebra and $V$ be an $H$-module for some Hopf algebra $H$.
We say that $(V, \psi)$, where $\psi \colon L \to \mathfrak{gl}(V)$,
 is a \textit{$(H,L)$-module} if $h(\psi(a)v)=\psi(h_{(1)}a)(h_{(2)}v)$
 for all $a\in L$, $h\in H$, $v\in V$.
   We say that an $(H,L)$-module $(V, \psi)$ is \textit{irreducible}
if it has no nontrivial $L$-submodules that are $H$-submodules at the same time.

\begin{theorem}\label{TheoremHWeyl}
Let $L$ be an $H$-(co)module Lie algebra over a field of characteristic $0$,
let $H$ be a finite dimensional (co)semisimple Hopf algebra,
and let $(V, \psi)$ be a finite dimensional $(H,L)$-module completely reducible as an $L$-module
disregarding the $H$-action.
 Then $$V=V_1 \oplus V_2 \oplus \ldots \oplus V_s$$
for some irreducible $(H,L)$-submodules $V_i$. 
\end{theorem}
\begin{proof}
Using the duality and combining Example~\ref{ExampleHSS} and Theorem~\ref{TheoremHcoWeyl}, we obtain the theorem.
\end{proof}
\begin{corollary}
Let $L$ be a finite dimensional semisimple $H$-(co)module Lie algebra over a field of characteristic $0$,
let $H$ be a finite dimensional (co)semisimple Hopf algebra,
and let $(V, \psi)$ be a finite dimensional $(H,L)$-module.
 Then $V=V_1 \oplus V_2 \oplus \ldots \oplus V_s$
for some irreducible $(H,L)$-submodules $V_i$. 
\end{corollary}

Let $L$ be a Lie $G$-algebra and let $V$ be an $FG$-module for some group $G$.
We say that $(V, \psi)$, where $\psi \colon L \to \mathfrak{gl}(V)$,
 is a \textit{$(G,L)$-module} if $g(\psi(a)v)=\psi(ga)(gv)$
 for all $a\in L$, $g\in G$, $v\in V$.
   We say that a $(G,L)$-module $(V, \psi)$ is \textit{irreducible}
if it has no nontrivial $G$-invariant $L$-submodules.

\begin{theorem}\label{TheoremAffAlgGrWeyl}
Let $L$ be a finite dimensional Lie algebra over an algebraically closed field $F$ of characteristic $0$
and let $G$ be a reductive affine algebraic group over $F$.
Suppose $L$ is endowed with a rational action of $G$ by automorphisms.
Let $(V, \psi)$ be a finite dimensional $(G,L)$-module with a rational $G$-action, completely reducible as an $L$-module disregarding the $G$-action.
 Then $$V=V_1 \oplus V_2 \oplus \ldots \oplus V_s$$
for some irreducible $(G,L)$-submodules $V_i$. 
\end{theorem}
\begin{proof}
We use Examples~\ref{ExampleAffAlgGr}, \ref{ExampleRatAffAlgGrAction}, and Theorem~\ref{TheoremHcoWeyl}.
\end{proof}
\begin{corollary}
Let $L$ be a finite dimensional semisimple Lie algebra over an algebraically closed field of characteristic $0$,
and let $G$ be a reductive affine algebraic group over $F$.
Suppose $L$ is endowed with a rational action of $G$ by automorphisms.
Let $(V, \psi)$ be a finite dimensional $(G,L)$-module with a rational $G$-action.
 Then $$V=V_1 \oplus V_2 \oplus \ldots \oplus V_s$$
for some irreducible $(G,L)$-submodules $V_i$. 
\end{corollary}

\section{$H$-(co)invariant decomposition of the solvable radical}\label{SectionHLBSN}

First we start with the case of not necessarily finite dimensional $H$.

\begin{theorem}\label{TheoremHcoLBSN}
Let $L$ be a finite dimensional $H$-comodule Lie algebra over a field $F$ of characteristic $0$
where $H$ is a Hopf algebra with an $\ad$-invariant integral $t \in H^*$, $t(1)=1$.
Suppose the nilpotent radical $N$ and the solvable radical $R$ are $H$-subcomodules.
Denote by $B$ a semisimple $H$-coinvariant subalgebra such that $L=B \oplus R$.
Then there exists an $H$-subcomodule $S$ such that $R = S \oplus N$
and $[B,S]=0$. Moreover, $L=B\oplus S\oplus N$.
\end{theorem}
\begin{proof}
Consider the adjoint representation of $B$ on $L$. Then $L$ is an $(H,B)$-module.
By the corollary of Theorem~\ref{TheoremHcoWeyl}, this module is completely reducible.
Moreover, the ideals $N$ and $R$ are $(H,B)$-submodules of $L$.
Thus there exists an $(H,B)$-submodule $S$ such that $R = S \oplus N$.
Hence $[B, S] \subseteq S$. 
However, by \cite[Proposition 2.1.7]{GotoGrosshans}, 
$[B, S] \subseteq [L,R]\subseteq N$. Thus $[B,S]=0$.
\end{proof}

Using the duality and applying Example~\ref{ExampleHSS} and Theorem~\ref{TheoremRadicals}, we derive from Theorem~\ref{TheoremHcoLBSN} the following

\begin{theorem}\label{TheoremHLBSN}
Let $L$ be a finite dimensional $H$-(co)module Lie algebra over a field $F$ of characteristic $0$
where $H$ is a finite dimensional (co)semisimple Hopf algebra.
Denote by $N$ the nilpotent radical and by $R$ the solvable radical of $L$.
Then there exists an $H$-sub(co)module $S$ such that $R = S \oplus N$,
$L=B\oplus S\oplus N$ (direct sum of $H$-sub(co)modules), and $[B,S]=0$
where $B$ is a maximal semisimple subalgebra.
\end{theorem}

In particular,

\begin{theorem}\label{TheoremGradLBSN}
Let $G$ be a finite group and let $L$ be a finite dimensional $G$-graded Lie algebra over a field $F$ of characteristic $0$. Denote by $N$ the nilpotent radical and by $R$ the solvable radical of $L$.
Then there exists a graded subspace $S$ such that $R = S \oplus N$,
$L=B\oplus S\oplus N$ (direct sum of graded subspaces), and $[B,S]=0$
where $B$ is a maximal semisimple subalgebra.
\end{theorem}

Also we obtain

\begin{theorem}\label{TheoremAffAlgGrLBSN}
Let $L$ be a finite dimensional Lie algebra over an algebraically closed field $F$ of characteristic $0$
and let $G$ be a reductive affine algebraic group over $F$.
Suppose $L$ is endowed with a rational action of $G$ by automorphisms.
Denote by $N$ the nilpotent radical and by $R$ the solvable radical of $L$.
Then there exists a $G$-invariant subspace $S$ such that $R = S \oplus N$,
$L=B\oplus S\oplus N$ (direct sum of $G$-invariant subspaces), and $[B,S]=0$
where $B$ is a maximal semisimple subalgebra.
\end{theorem}
\begin{proof}
We notice that $R$ and $N$ are invariant under all automorphisms, and 
use Examples~\ref{ExampleAffAlgGr}, \ref{ExampleRatAffAlgGrAction}, and Theorem~\ref{TheoremHcoLBSN}.
\end{proof}

\section*{Acknowledgements}

I am grateful to Yuri Bahturin and Mikhail Kotchetov for helpful
discussions.

\end{document}